\documentclass[12pt]{amsart}

\usepackage[T1]{fontenc}
\usepackage[utf8]{inputenc} 
\usepackage[english]{babel}	

\usepackage{amsmath} 
\usepackage{amssymb} 
\usepackage{mathtools} 
\usepackage{amsthm}	
\usepackage{amsfonts} 
\usepackage{tikz-cd} 
\usepackage{eufrak} 

\usepackage{enumerate}

\usepackage{algorithmic}
\usepackage[ruled]{algorithm}
\usepackage{caption}

\usepackage{ulem}

\title[A footnote to a paper of Deodhar]
{A footnote to a paper of Deodhar}

\author{Davide Franco }
\address{Universit\`a di Napoli
\lq\lq Federico II\rq\rq, Dipartimento di Matematica e
Applicazioni \lq\lq R. Caccioppoli\rq\rq, Via Cintia, 80126
Napoli, Italy.} \email{davide.franco@unina.it}

\theoremstyle{plain}
\newtheorem{theorem}{Theorem}[section]
\newtheorem{corollary}[theorem]{Corollary}

\newtheorem{proposition}[theorem]{Proposition}

\theoremstyle{definition}

\newtheorem{remark}[theorem]{Remark}
\newtheorem{notations}[theorem]{Notations}
\newtheorem{definition}[theorem]{Definition}

\newcommand{\uc}{\ensuremath{\mathcal{U}}}

\newcommand{\dc}{\ensuremath{\mathcal{D}}}

\newcommand{\hc}{\ensuremath{\mathcal{H}}}

\newcommand{\lc}{\ensuremath{\mathcal{L}}}

\newcommand{\bC}{\mathbb{C}} 

\newcommand{\Z}{\mathbb{Z}}
\newcommand{\Zq}{\mathbb{Z}[q^{\frac{1}{2}}, q^{-\frac{1}{2}}]}

\begin{document}

\begin{abstract} 

Let $X\subseteq G\slash B$ be a Schubert variety in a flag manifold and let $\pi: \tilde X \rightarrow X$ be a Bott-Samelson resolution of $X$.
In this paper we prove  an effective version of  the decomposition theorem for the  derived pushforward $R \pi_{*} \mathbb{Q}_{\tilde{X}}$.
As a by-product,  we obtain recursive procedure to extract Kazhdan-Lusztig polynomials  from   the polynomials introduced by V. Deodhar in \cite{Deo}, which does not require prior knowledge of a minimal set.
We also observe that any family of equivariant resolutions of Schubert varieties allows to define a new basis in the Hecke algebra and we show a way to compute the transition matrix, from the   Kazhdan-Lusztig basis to the new one.


\end{abstract}

\maketitle

\section{Introduction}

As the title suggests, this work is a sort of appendix to  \cite{Deo}. In such a paper,  Vinay Deodhar introduces a statistic, called \textit{defect}, on the subexpressions
of a given reduced expression of an element of a Coxeter group $W$ (see also \cite[Section 6.3.17]{BillLak}). Specifically, 
let $w\in W$ be an element of length $l(w)$ and let
$$w=s_1 \dots s_l, \quad l=l(w)$$
be a reduced expression of $w$.
A \textit{subexpression} $\sigma=(\sigma_0, \dots , \sigma_l)$ is a sequence of Coxeter group elements such that
$$\sigma_{j-1}^{-1}\sigma_j \in \{id, s_j\}, \quad \text{for all} \quad 1\leq j \leq l. $$
Let $\mathcal{S}$ be the set such sequences for the given reduced expression and let
$$\pi(\sigma):= \sigma_l, \quad \text{if} \quad \sigma=(\sigma_0, \dots , \sigma_l)\in\mathcal{S}. $$
For any subexpression $\sigma=(\sigma_0, \dots , \sigma_l)\in\mathcal{S}$, Deodhar defines the \textit{defect} of $\sigma$ by
$$ d(\sigma):= \# \{1\leq j \leq l \mid \sigma_{j-1}^{-1}s_j < \sigma _{j-1} \}.$$
If one  fix a reduced expression for all $w\in W$ and consider $v\in W$ such that $v\leq w$ in the Bruhat order, then one can use the defect to define the following polynomial
\begin{equation}
\label{defdeo}
Q_{w, v}:= \sum _{\sigma \in \mathcal S, \, \pi(\sigma)=v}q^{d(\sigma)} \in \Z[q].
\end{equation}

In \cite{Deo}, Deodhar proves that the Kazhdan-Lusztig polynomial $P_{w, v}$ admit a description  as subsum of \eqref{defdeo}. More precisely, he gives a recursive algorithm for computing a minimal set $E_{min}\subseteq \mathcal S$ such that
\begin{equation}
\label{defKL}
P_{w, v}:= \sum _{\sigma \in E_{min}, \, \pi(\sigma)=v}q^{d(\sigma)} \in \Z[q],
\end{equation}
for any pair $w, v\in W$, $v\leq w$.
What is more, in \cite{Deo} one can find a new basis of the Hecke algebra of $W$ that is defined starting from the polynomials \eqref{defdeo}.

The main aim of this paper is to give a recursive procedure to extract Kazhdan-Lusztig polynomials $P_{w, \, v}$ from  polynomials $Q_{w, \, v}$, without going through the computation of the minimal set $E_{min}\subseteq \mathcal S$, in the case  $W$ is the Weyl group of a semisimple connected algebraic group over  $\bC$.
Instead, our approach is based on an effective version of the Beilinson-Bernstein-Deligne-Gabber decomposition theorem (BBDG for short) for the Bott-Samelson resolution. As a by-product of our analysis, we provide a recursive procedure to compute  the change-of-basis matrix, from the   Kazhdan-Lusztig basis of the Hecke algebra to the basis defined in \cite{Deo}.
 
The starting point of our analysis is Proposition 3.9 of
 \cite{Deo},  where the author shows that, when $W$ is the Weyl group of a semisimple connected algebraic group,  the polynomial \eqref{defdeo} has a nice geometric interpretation as Poincar\'e polynomial of a suitable fiber of the \textit{Bott-Samelson resolution}  of the Schubert variety $X(w)$ (see section (2) for a short review of some standard definitions and notations concerning Schubert varieties). More precisely, if we denote by $G$ an algebraic group with Weyl group $W$ and Borel subgroup $B$,  then  to the chosen reduced expression
$w=s_1 \dots s_l$ it is also associated the Bott-Samelson resolution
$$\pi_{w}: \tilde X (w) \rightarrow X(w).$$
The smooth variety $\tilde X (w)$ is defined as the subvariety of $(G\slash B)^l$ consisting of $l$-tuples
$(g_1B, \dots, g_lB)$ such that
$$ g_{i-1}^{-1}g_i\in \overline{Bs_iB}, \quad 1\leq i \leq l\quad (\text{by convention}, \,\,g_0=1),$$
and $\pi_w$ is the projection on the last factor. 

The polynomial \eqref{defdeo} is the Poincar\'e polynomial of the fiber of $\pi_w$ over the cell $\Omega_v\subset X(w)$ associated to $v$:
\begin{equation}\label{defQintro}
Q_{w, \, v}= \sum _i \dim H^{2i}(\pi_w^{-1}(x))q^i, \quad \forall x\in \Omega (v).
\end{equation}

Our approach for extracting the Kazhdan-Lusztig polynomials  from  the polynomials  defined in \eqref{defdeo}, is to prove an effective version of the  BBDG decomposition theorem for the Bott-Samelson resolution and, more generally, for any  equivariant resolution of a Schubert variety (in \cite{dGeFr2019}, \cite{Forum},  and \cite{CiFrSe}  partial results in this direction were previously obtained).
Specifically, let  $D_{c}^{b}(X)$ be the derived category of bounded complexes of constructible $\mathbb{Q}$-vector sheaves on a Schubert variety $X\subseteq G\slash B$.  The  decomposition theorem applied to an equivariant  resolution 
$\pi:\tilde X \rightarrow X$,
 states that 
the derived direct image
$R \pi_{*} \mathbb{Q}_{\tilde{X}} [ \dim X ]$  splits in $D_{c}^{b}(X)$  as a direct sum of shifts of irreducible perverse sheaves on $X$.
By \cite[\textsection~1.5]{dCaMi2009}, we have  a non-canonical decomposition
\begin{equation}\label{decThmintr}
	R \pi_{*} \mathbb{Q}_{\tilde{X}} [ \dim X] \cong \bigoplus_{i \in \mathbb{Z}} \bigoplus_{j \in \mathbb{N}} IC(L_{ij}) [-i],
\end{equation}
where the summands are shifted intersection cohomology complexes of the semisimple local systems $L_{ij}$, each of which is supported on a suitable locally closed stratum of codimension $j$, usually called a \textit{support} of the decomposition. The summand supported in the general point is precisely the intersection cohomology of $X$. The supports appearing in the splitting~\eqref{decThmintr} and the local systems $L_{ij}$ are, generally, rather mysterious objects when $j \geq 1$.

Quite luckily, in our case a crucial simplification arises because all the local systems $L_{ij}$ appearing in the decomposition (\ref{decThmintr}) are trivial by an easy argument that is explained in Proposition \ref{WK}.
As a consequence, we have 
\begin{equation}\label{DTintro}
	R \pi_{*} \mathbb{Q}_{\tilde{X}} [ \dim X] \cong \bigoplus_{X(v)\subseteq X} \bigoplus_{\alpha \in \mathbb{Z}} IC_{X(v)}^{\oplus s_{v, \alpha}},
\end{equation}
for suitable multiplicities $s_{v, \alpha}$.
Since we have a splitting like \eqref{DTintro} for any equivariant resolution $\pi_w: \tilde{X}(w) \rightarrow X(w)$, we 
can define  a Laurent polynomial recording the contribution to the decomposition theorem  for $\pi_w$, with support $X(v)$:
\begin{equation*}
		D_{w,  v} (t) := \sum_{\alpha \in \mathbb{Z}} s_{v, \alpha}^{w} \cdot t^{\alpha}\in \Z[t, t^{-1}],
	\end{equation*}
for avery pair $(v,w)$ in $W$, such that $v\leq w$, and for every resolution.

Now, assume to have fixed an equivariant resolution $\pi_w: \tilde{X}(w) \rightarrow X(w)$ for every Schubert variety $X(w)$ and assume that the cohomology of the fibers of $\pi_w$ vanish in odd degrees (this property is satisfied by any reasonable resolution of Schubert varieties).
Similarly as in \eqref{defQintro}, define the analogue Deodhar's polynomial $Q_{w,  v}$ as the Poincar\'e polynomial 
of the fibers of $\pi_w$ over the cell $\Omega(v)$.
The main results contained in this paper can be summarized as follows:

a) we set up an iterative procedure  that allows to compute  \textbf{both} the Kazhdan-Lusztig polynomials $P_{w, v}$ and the Laurent polynomials $D_{w,  v} $ starting from  Deodhar's polynomials $Q_{w,  v}$;

b) we observe that the polynomials  $Q_{w,  v}$ allow to construct a new basis $\{B_w \mid  w\in W\}$ of the Hecke algebra;

c) we prove that the transition matrix, from the   Kazhdan-Lusztig basis to the new basis $\{B_w \mid  w\in W\}$, can be easily deduced from the Laurent polynomials $D_{w,  v}$ \textit{and does not require prior knowledge of the transition matrix from the Kazhdan-Lusztig basis to the standard one}
(compare with Remark \ref{last}).

\bigskip
\section{Notations and basic facts}

\medskip

In this section we review some basic facts concerning  Buhat decomposition, Schubert varieties and combinatorics of subexpressions that are needed in the following.
\medskip
\par\noindent
(i) Let $G$ be a semisimple connected algebraic group over  $\bC$.
Let $T$ and $B$ be a \textit{maximal torus} and a \textit{Borel  subgroup} of $G$, respectively. Denote by $W$ be the \textit{Weyl group} of $G$.
If we consider
$$\{e_w \mid \,\, w\in W\}\subset G\slash B,$$ 
the set of fixed points for the torus action on $G\slash B$,
then we have the \textit{Bruhat decomposition} of $G\slash B$ i.e. the disjoint union of \textit{Bruhat cells}
$$G\slash B=  \bigsqcup_{w\in W}\Omega(w), \quad \Omega(w):=Be_w.$$
For every $w\in W$ the \textit{Schubert variety} associated to $w$  is defined as the Zariski closure of the corresponding Bruhat cell:
$$X(w):= \overline{\Omega(w)}.$$
(ii) There is a partial order on the Weyl group $W$ determined by the decomposition above. Specifically, for $w_1, w_2\in W$ we have
$$w_1\geq w_2 \quad \Leftrightarrow \quad X(w_1) \supseteq X(w_2).$$ 
Furthermore,  we have
$$X(w)=  \bigcup_{v \leq w}\Omega(v).$$

\medskip
We borrow from Deodhar's paper \cite{Deo} some crucial definitions concerning the combinatorics of subexpressions of a reduced word in a Coxeter group.

\begin{definition} \cite[Def. 2.1 - 2.2]{Deo}
\begin{enumerate}
\item Let $w\in W$ be an element of length $l(w)$ and let
$$w=s_1 \dots s_l, \quad l=l(w)$$
be  a reduced expression of $w$.
A \textit{subexpression} $\sigma=(\sigma_0, \dots , \sigma_l)$ is a sequence of Weyl group elements such that $\sigma_0=id$ and
$$\sigma_{j-1}^{-1}\sigma_j \in \{id, s_j\}, \quad \text{for all} \quad 1\leq j \leq l. $$
For any reduced word $r$, let $\mathcal{S} _r$ be the set of reduced expressions of $r$ and let
$$\pi(\sigma):= \sigma_l, \quad \text{if} \quad \sigma=(\sigma_0, \dots , \sigma_l)\in\mathcal{S} _r. $$
\item For any $\sigma=(\sigma_0, \dots , \sigma_l)\in\mathcal{S} _r$ define the \textit{defect} of $\sigma$ by
$$ d(\sigma):= \# \{1\leq j \leq l \mid \sigma_{j-1}^{-1}s_j < \sigma _{j-1} \}.$$
\end{enumerate}
\end{definition}
\medskip
\par\noindent
If $v\leq w$, we let
\begin{equation}
\label{defQ}
Q_{w, v}:= \sum _{\sigma \in \mathcal S_r, \,\pi(\sigma)=v}q^{d(\sigma)} \in \Z[q].
\end{equation}
\medskip
\textit{From now on we assume that we have chosen a reduced expression for all $w\in W$}, hence \eqref{defQ} provides a polynomial $Q_{w, v}\in \Z[q]$ for any pair $(w,v)$ such that $v\leq w$.
\medskip

As explained in \cite[Proposition 3.9]{Deo}, the polynomial above has a nice geometric interpretation as Poincar\'e polynomial of the \textit{Bott-Samelson resolution}. We recall that to each reduced expression
$w=s_1 \dots s_l,$ it is also associated the \textit{Bott-Samelson resolution}
$$\pi_{w}: \tilde X (w) \rightarrow X(w).$$
The smooth variety $\tilde X (w)$ is defined as the subvariety of $(G\slash B)^l$ consisting of $l$-tuples
$(g_1B, \dots, g_lB)$ such that
$$ g_{i-1}^{-1}g_i\in \overline{Bs_iB}, \quad 1\leq i \leq l\quad (\text{by convention}, \,\,g_0=1),$$
and $\pi_w$ is the projection on the last factor. It is clear that $\pi_w$ is equivariant under the action of the Borel subgroup $B$.

By  \cite[Proposition 3.9]{Deo}, \eqref{defQ} is the Poincar\'e polynomial of the fiber of $\pi_w$ over the cell $\Omega(v)\subset X(w)$:
\begin{equation}
\label{defQPoi}
Q_{w, v}= \sum _i \dim H^{2i}(\pi_{w}^{-1}(x))q^i, \quad \forall x\in \Omega(v).
\end{equation}

\bigskip
\bigskip

\section{The Decomposition theorem}

\medskip

Formula \eqref{defQPoi} and Deodhar's paper suggest that there should be a closed relationship between the Poincar\'e polynomials of the stalk cohomology of the  complex $R \pi_{*} \mathbb{Q}_{\tilde{X}(w)}$ and the Kazhdan-Lusztig polynomials.
The most important result concerning the complex $R \pi_{*} \mathbb{Q}_{\tilde{X}(w)}$  and, in general, concerning the topology of proper algebraic  map is the \textit{decomposition theorem} of Beilinson, Bernstein, Deligne and Gabber,
which we now recall.

\bigskip
In what follows, we shall work  cohomology with $\mathbb{Q}$-coefficients  and the  self-dual perversity $\mathfrak{p}$ (see \cite[\textsection 2.1]{BBD},  and \cite[p. 79]{GMP2}). 

\begin{theorem}(\textbf{Decomposition theorem} \textup{\cite[1.6.1]{dCaMi2009}})\label{ThDec}
	Let $f: X \rightarrow Y$ be a proper map of complex algebraic varieties.  In $D_{c}^{b}(Y)$, the derived category of bounded complexes of constructible $\mathbb{Q}$-vector sheaves on  $Y$, there is a non-canonical isomorphism
	\begin{equation}\label{Eq:DecTh01}
		Rf_{*} IC_{X} \cong \bigoplus_{\alpha \in \mathbb{Z}} \prescript{\mathfrak{p}}{}{\mathcal{H}}^{\alpha}(Rf_{*} IC_{X}) \left[ - \alpha \right].
	\end{equation}
	Furthermore, the perverse sheaves $\prescript{\mathfrak{p}}{}{\mathcal{H}}^{\alpha}(Rf_{*} IC_{X})$ are semisimple; i.e.~there is a decomposition into finitely many disjoint locally closed and nonsingular subvarieties $Y = \coprod S_{\beta}$ and a canonical decomposition into a direct sum of intersection complexes of semisimple local systems
	\begin{equation}\label{Eq:DecTh02}
		\prescript{\mathfrak{p}}{}{\mathcal{H}}^{\alpha}(Rf_{*} IC_{X}) \cong \bigoplus_{\beta} IC_{\overline{S_{\beta}}}(L_{\alpha, S_{\beta}}).
	\end{equation}
\end{theorem}

Combining \eqref{Eq:DecTh01} and \eqref{Eq:DecTh02} we have
\begin{equation}
	Rf_{*} IC_{X} \cong \bigoplus_{\alpha \in \mathbb{Z}} \prescript{\mathfrak{p}}{}{\mathcal{H}}^{\alpha}(Rf_{*} IC_{X}) [- \alpha] \cong \bigoplus_{\alpha \in \mathbb{Z}} \bigoplus_{\beta} IC_{\overline{S_{\beta}}}(L_{\alpha, S_{\beta}}) [- \alpha],
\end{equation}
which can be written in the form
\begin{equation*}
	Rf_{*} IC_{X} \cong \bigoplus_{\alpha \in \mathbb{Z}} \prescript{\mathfrak{p}}{}{\mathcal{H}}^{\alpha}(Rf_{*} IC_{X}) [- \alpha] \cong \bigoplus_{\alpha \in \mathbb{Z}} \bigoplus_{S} \prescript{\mathfrak{p}}{}{\mathcal{H}}^{\alpha}(Rf_{*} IC_{X})_{S} [- \alpha],
\end{equation*}
where $S$, called a \textbf{support} of $f$, is any $\overline{S_{\beta}}$ associated to a non-zero local system $L_{\alpha, S_{\beta}}$ (see \cite[Definition~9.3.41]{Max2019}).
In the literature
one can find different approaches to the Decomposition Theorem
(see \cite{BBD}, \cite{Sai1986}, \cite{dCaMi2009},   
\cite{Wil2017}), which is a very general result but also rather implicit.
On the other hand,  there are many special cases for which the Decomposition
Theorem admits a simplified and explicit approach. One of these is  the case of
varieties with isolated singularities. For instance, in the work  \cite{dGeFr2020}, a simplified approach to the Decomposition Theorem for varieties with isolated singularities is developed, in connection with the existence of a \textit{natural Gysin morphism}, as defined in \cite[Definition 2.3]{dGeFr2017OnTheExistence} (see also \cite{DGF2} for other applications of the Decomposition Theorem to the Noether-Lefschetz Theory).

\medskip

As remarked before, the Bott-Samelson resolution $$\pi_{w}: \tilde X (w) \rightarrow X(w),$$ is equivariant under the action of the Borel subgroup $B$, hence $\pi_{w}$ is stratified according to the Bruhat decomposition
\begin{equation*}
X(w) = \bigsqcup_{v\leq w} \Omega(v).
\end{equation*}
In this case,  the supports of the decomposition theorem applied to the resolution $\pi_w$ are the Schubert subvarieties
\begin{equation*}
	X(v)=\overline{\Omega}(v), \qquad v \leq w.
\end{equation*}

Furthermore, all local systems appearing in the decomposition theorem are trivial since the isotropy subgroup of each orbit $\Omega(v)$ is connected \cite[Remark 11.6.2]{HTT}. 

We include in  the following proposition a proof of these facts, although they are probably well-known, in the attempt of making the present paper reasonably self-contained and also because the simple argument is very close to the rest of the paper.

\begin{proposition}\label{WK}
Let $\pi: \tilde X \rightarrow X$ be an equivariant     resolution of a Schubert variety $X=X(w)$ of dimension $l$. In the derived category $D_{c}^{b}(X)$,
we have a splitting
\begin{equation*}
		R \pi _{*} (\mathbb{Q}_{\tilde{X}}) [l] =  \bigoplus_{v \leq w} \bigoplus_{\alpha \in \mathbb{Z}} IC_{X(v)}^{ \oplus s_{v, \alpha}} [- \alpha],
	\end{equation*}
	for suitable multiplicities $s_{v, \alpha}$.
In other words, 
the supports of the decomposition theorem applied to the resolution $\pi$ are the Schubert varieties
 contained in $X$
and all local systems  are trivial. 
\end{proposition}
\begin{proof}
Let $l:=l(w)$, fix $r$ such that $-1\leq r\leq l$ and define the following decreasing sequence of open sets of $X$
$$\ \uc_r:= X(w)\backslash  \bigsqcup_{v\leq w, \, l(v)\leq r} \Omega(v). $$
Clearly  we have $\uc_l= \emptyset$ \and $\uc _{-1}=X$. We are going to prove, by decreasing induction on $r$, that 
\begin{equation}\label{ththeorem}
		R \pi _{*} (\mathbb{Q}_{\tilde{X}}) [l] \mid_{\uc _r}=  \bigoplus_{v \leq w,\, r<l(v)} \bigoplus_{\alpha \in \mathbb{Z}} IC_{X(v)}^{ \oplus s_{v, \alpha}} [- \alpha]\mid_{\uc _r},
	\end{equation}
for suitable multiplicities $s_{v, \alpha}$.

Since $\uc_{l-1}=\Omega:=\Omega (w)$ and since $\pi$ is an isomorphism over $\Omega$,  we have
$$R \pi _{*} \mathbb{Q}_{\tilde{X}} [l] \mid_{\Omega}\cong \mathbb{Q}_{\Omega} [l] ,$$ 
hence the \textit{base step} follows from the well known isomorphism
$$ \mathbb{Q}_{\Omega} [l] \cong IC_{X}\mid_{\Omega}$$ 
(compare e.g. with  \cite[Definition 6.3.1]{Max2019}).
\par\noindent
As for the \textit{inductive step}, let 
$$\dc:= \bigsqcup_{v \leq w,\, r=l(v)}\Omega(v).$$
 By induction, we have
$$R \pi_{*} (\mathbb{Q}_{\tilde{X}}) [l] \mid_{\uc _{r}}=  \bigoplus_{v \leq w,\, r<l(v)} \bigoplus_{\alpha \in \mathbb{Z}} IC_{X(v)}^{ \oplus s_{v, \alpha}} [- \alpha]\mid_{\uc _{r}},$$
hence we deduce
\begin{equation}\label{ththeoremL}
		R \pi_{*} (\mathbb{Q}_{\tilde{X}}) [l] \mid_{\uc _{r-1}}=  \lc\mid_{\dc} \oplus \bigoplus_{v \leq w,\, r<l(v)} \bigoplus_{\alpha \in \mathbb{Z}} IC_{X(v)}^{ \oplus s_{v, \alpha}} [- \alpha]\mid_{\uc _{r-1}},
	\end{equation}
where $\lc$ gathers all the summands supported in $\uc_r^*=\bigsqcup_{v \leq w,\, l(v)\leq r}\Omega(v)$
$$\lc:= \bigoplus_{\alpha \in \mathbb{Z}} \bigoplus_{S\subseteq \uc_r^*} \prescript{\mathfrak{p}}{}{\mathcal{H}}^{\alpha}(R \pi_{*} (\mathbb{Q}_{\tilde{X}}) [l] )_{S} [- \alpha].$$
In the previous formula, the summand $\prescript{\mathfrak{p}}{}{\mathcal{H}}^{\alpha}(R \pi_{*} (\mathbb{Q}_{\tilde{X}}) [l] )_{S}$ denotes the $S$ component of $\prescript{\mathfrak{p}}{}{\mathcal{H}}^{\alpha}(R \pi_{*} (\mathbb{Q}_{\tilde{X}}) [l] )$ in the decomposition by supports
\cite[Section 1.1]{DeC}.
 By proper base change we also have
\begin{equation}\label{ththeoremL+}
R \pi_{*} (\mathbb{Q}_{\tilde{X}}) [l] \mid_{\dc}=  \lc\mid_{\dc} \oplus \bigoplus_{v \leq w,\, r<l(v)} \bigoplus_{\alpha \in \mathbb{Z}} IC_{X(v)}^{ \oplus s_{v, \alpha}} [- \alpha]\mid_{\dc}.
\end{equation}
Since $\pi: \tilde X \rightarrow X$ is equivariant and $\dc$ is a disjoint union of $B$-orbits, the dimension of the cohomology stalk
$\hc^i R \pi_{*} (\mathbb{Q}_{\tilde{X}}) [l]_x$ is independent of $x\in \dc$, for all $i$. The same holds true also for all $IC_{X(v)} [- \alpha]\mid_{\dc}$.
Then \eqref{ththeoremL+} shows that the dimension of the cohomology stalk $\hc^i \lc_x$ is independent of $x\in \dc$, for all $i$. Thus $\lc\mid_{\dc}$ is
a direct sum of shifted local systems because,
by \cite[Remark 1.5.1]{dCaMi2009}, the perverse cohomology sheaves $\prescript{\mathfrak{p}}{}{\mathcal{H}}^{i}(\lc\mid_{\dc})$ concide, up to a shift, with the ordinary cohomology
$$\prescript{\mathfrak{p}}{}{\mathcal{H}}^{i}(\lc\mid_{\dc})\cong \hc^{i-r}\lc\mid_{\dc}[r],\quad \forall i.$$
We are done, because 
$\dc= \bigsqcup_{v \leq w,\, r=l(v)}\Omega(v)$ is a disjoint union of affine spaces of dimension $r$ (compare e.g. with \cite[Theorem 9.9.5 (i)]{HTT})  so any local system on $\dc$ is trivial and \eqref{ththeorem} follows for the restriction to the open set $\uc_{r-1}$.

\end{proof}

\bigskip

\section{A consequence of the decomposition theorem}

\medskip
In this section we assume to have fixed an equivariant resolution
$\pi_{w}: \tilde X (w) \rightarrow X(w),$ for any Schubert variety $X(w)$, $w\in W$.
As a consequence of Proposition \ref{WK}, the decomposition theorem for $\pi_w$ can be stated as
\begin{equation}\label{decEx}
		(R \pi_w) _{*} \mathbb{Q}_{\tilde{X}(w)} [l(w)] =  \bigoplus_{v \leq w} \bigoplus_{\alpha \in \mathbb{Z}} IC_{X(v)}^{ \oplus s_{v, \alpha}^{w}} [- \alpha],
	\end{equation}
for suitable multiplicities $s_{v, \alpha}^w$, where recall that $l(w) = \dim X(w)$.

\begin{notations}
	For any pair $(v,w)$ of permutations such that $v \leq w$, let
	\begin{equation}\label{defD}
		D_{w, v} (t) := \sum_{\alpha \in \mathbb{Z}} s_{v, \alpha}^{w} \cdot t^{\alpha} \in \Z[t, t^{-1}]
	\end{equation}
	be the Laurent polynomial recording  the contribution to the decomposition theorem \eqref{decEx} coming from support $X(v)$.
Let moreover 
\begin{equation}\label{defF}
		F_{w, v} (t) := \sum_{\alpha \in \mathbb{Z}} f_{w, v}^{l(w)+\alpha} t^{\alpha} \in \Z[t, t^{-1}], \quad f_{w, v}^{i}  := \dim H^{i}(\pi_w^{-1}(x)), \,\,\,x\in \Omega(v),
	\end{equation}
	be the shifted Poincar\'e polynomial of the fibers of $\pi_w$ over $\Omega(v)$.
\end{notations}

\begin{remark}
	Let $X(w)$ be a Schubert variety and let $\Omega(v)\subseteq X(w)$ be a Schubert cell. It is well known the dimensions of the stalks $\mathcal{H}^{\alpha}(IC_{X(w)})_{x}$ do not depend on $x\in \Omega (v)$. 
	
The Laurent polynomial encoding the dimensions of the stalks $\mathcal{H}^{\alpha}(IC_{\mathcal{S}_{\tau}}^{\bullet})_{x}$ is the shifted Kazhdan-Lusztig polynomial:
	\begin{equation*}
		H_{w, v}(t) := \sum_{\alpha \in \mathbb{Z}} h_{w, v}^{\alpha} t^{\alpha}, \quad h_{w, v}^{\alpha}  := \dim \mathcal{H}^{\alpha}(IC_{X(w)})_{x}, \,\,\, x \in \Omega (v).
	\end{equation*}
Recall that we have	
\begin{equation}\label{KvsH}
		P_{w, v}(q)= q^{\frac{l(w)}{2}}H_{w, v}(\sqrt q), 
	\end{equation}
where $P_{w, v}(q)$ is the corresponding Kazhdan-Lusztig polynomial \eqref{defKL} (compare e.g. with \cite[Theorem 6.1.11]{BillLak}).
\end{remark}
Before stating the main result of this section, let us introduce the truncation $U$ and symmetrizing $S$ operators:
\begin{align*}
	&U: \sum_{\alpha \in \mathbb{Z}} a_{\alpha} t^{\alpha} \in \mathbb{Z}[t, t^{-1}] \mapsto \sum_{\alpha \geq 0} a_{\alpha} t^{\alpha}\in \mathbb{Z}[t];\\
	&S: \sum_{\alpha \geq 0} a_{\alpha} t^{\alpha}\in \mathbb{Z}[t] \mapsto a_{0} + \sum_{\alpha > 0} a_{\alpha} (t^{\alpha} + t^{- \alpha}) \in \mathbb{Z}[t, t^{-1}].
\end{align*}

\begin{theorem}\label{Th:Main}
	With notations as above, let $u\leq w$ in $W$. Then we have the following  recursive formulae  for the computation of the Laurent polynomials $D_{w, u}$ and the shifted Kazhdan-Lusztig polynomials~$H_{w, u}$:
	\begin{equation*}
		\begin{cases}
			D_{w, u} = S \circ U(R_{w, u})\\
			H_{w, u} = t^{-l(u)}(R_{w, u} - D_{w, u})
		\end{cases}
	\end{equation*}
	where
	\begin{equation*}
		R_{w, u} := t^{l(u)} \left( F_{w, u} - \sum_{u < v < w} D_{w, v} \cdot H_{v, u} \right).
	\end{equation*}
\end{theorem}
\begin{proof}
For the sake of simplicity, in the proof we set $\pi: \tilde X \rightarrow X$ instead of \par \noindent
$\pi_w: \tilde{X} (w)\rightarrow X(w)$.
By Proposition \ref{WK}, we have
\begin{equation}\label{decExp}
		R \pi _{*} \mathbb{Q}_{\tilde{X}} [l] =  \bigoplus_{v \leq w} \bigoplus_{\alpha \in \mathbb{Z}} IC_{X(v)}^{ \oplus s_{v, \alpha}^{w}} [- \alpha].
	\end{equation}

Consider a cell $\Omega(u)\subset X$, take  the stalk cohomology at $x\in \Omega(u)$ and recall \eqref{defD} and \eqref{defF}. From \eqref{decExp} we infer
	\begin{equation}\label{Eq:Eq1MainTh}
		F_{w, u}(t)= \sum_{u\leq v \leq w} D_{w, v}(t) \cdot H_{v, u}(t)=
	\end{equation}
	$$D_{w, w}(t) \cdot H_{w, u}(t)+ D_{w, u}(t) \cdot H_{u, u}(t)+ \sum_{u< v < w} D_{w, v}(t) \cdot H_{v, u}(t).$$
Since the resolution $\pi: \tilde X \rightarrow X$ is equivariant, it must be an ismorphism over $\Omega(w)$ and we have $D_{w, w}(t) =1$ (recall \eqref{decEx} and \eqref{defD}). Furthermore, from the well known isomorphism $IC_{X(u)}|_{\Omega(u)} \cong \mathbb{Q}_{\Omega(u)}[l(u)]$ 
(compare e.g. with \cite[Definition 6.3.1]{Max2019})
we deduce $H_{u, u}(t)=t^{-l(u)}$. Hence, from \eqref{Eq:Eq1MainTh} we get
\begin{equation}\label{Eq:Eq1MainTh}
		F_{w, u}(t)=  H_{w, u}(t)+ t^{-l(u)}\cdot D_{w, u}(t) + \sum_{u< v < w} D_{w, v}(t) \cdot H_{v, u}(t).
	\end{equation}

The last equality can be written as
	\begin{equation*}
		D_{w, u}(t) = t^{l(u)}\cdot \left( F_{w, u}(t) - \sum_{u< v < w} D_{w, v}(t) \cdot H_{v, u}(t)\right) - t^{l(u)} \cdot H_{w, u}(t) =  R_{w, u} - t^{l(u)} \cdot H_{w, u}(t).
	\end{equation*}
	The support conditions for perverse sheaves imply that $t^{l(u)} \cdot H_{w, u}(t)$ is concentrated in negative degrees (see \cite[p. 552, equation 12]{dCaMi2009}), thus
	\begin{equation*}
		U(D_{w, u}(t)) = U(R_{w, u}(t)).
	\end{equation*}
	Finally, the Laurent polynomials $D_{w, u}(t)$ are symmetric because of Hard-Lefschetz theorem (see \cite[Theorem 1.6.3]{dCaMi2009}), that is to say
	\begin{equation*}
		D_{w, u}(t)= D_{w, u}(t^{-1}).
	\end{equation*}
	Therefore we have
	\begin{equation}\label{eq:equazioni del teorema}
		\begin{cases*}
			D_{w, u}(t) = S \circ U(R_{w, u}(t))\\
			H_{w, u}(t) = t^{-l(u)}(R_{w, u}(t)-D_{w, u}(t))
		\end{cases*}
	\end{equation}
	and the statement follows.
\end{proof}

\medskip
By \eqref{KvsH}, the last theorem provides an iterative procedure to compute Kazhdan-Lusztig polynomials $P_{w, v}$ from Poincar\'e polynomials $Q_{w, v}$. 
To this end, let us introduce the following operators:
\begin{align*}
	U_{\beta}&: \sum_{\alpha \geq 0} c_{\alpha} t^{\alpha} \in \mathbb{Z} \left[ t \right] \mapsto \sum_{\alpha \geq \beta} c_{\alpha} t^{\alpha} \in \mathbb{Z} \left[ t \right], \hspace{0.33em} \forall \beta \geq 0.\\
\end{align*}

Next statement  follows from Theorem \ref{Th:Main} and collects all informations we obtained until now.

\begin{corollary}\label{cor:Utilde} Assume to have fixed
an equivariant resolution
$\pi_{w}: \tilde X (w) \rightarrow X(w),$ for any Schubert variety $X(w)$, $w\in W$.
For any pair $w, u$ in $W$ such that $u\leq w$, let $\tilde F_{w, u}$ be the Poincar\'e polynomials of  the fiber $\pi_w^{-1}(x)$, $\forall x\in \Omega(u)$. 
Then we have the following  recursive formulae:
	\begin{equation*}
		\begin{cases}
			\tilde D_{w, u} = t^{l(w)-l(u)} \circ S \circ t^{l(u)-l(w)} \circ  U_{l(w)-l(u)}(\tilde R_{w, u})\\
			\tilde H_{w, u} = \tilde R_{w, u} -\tilde D_{w, u}
		\end{cases}
	\end{equation*}
	where
	\begin{equation*}
\tilde D_{w, u}:= t^{l(w)-l(u)}  D_{w, u}, \quad \tilde H_{w, u}:= t^{l(w)}  H_{w, u}, \quad	\tilde R_{w, u} := \tilde F_{w, u} - \sum_{u < v < w} \tilde D_{w, v} \cdot \tilde H_{v, u}. 
	\end{equation*}
\end{corollary}
\begin{proof}
From \eqref{defF} we find $\tilde F_{w, u}=t^{l(w)}F_{w, u}$. Thus we have 
$$	\tilde R_{w, u} := \tilde F_{w, u} - \sum_{u < v < w} \tilde D_{w, v} \cdot \tilde H_{v, u}= 	t^{l(w)} F_{w, u} - \sum_{u < v < w}  t^{l(w)-l(u)}D_{w, v} \cdot t^{l(v)} H_{v, u}=t^{l(w)-l(v)} R_{w, u},$$
and the statement  straightforwardly follows just
combining Theorem \ref{Th:Main} with
$$ t^{l(u)-l(w)} \circ U_{l(w)-l(u)}\circ t^{l(w)-l(v)}= U_0.$$
\end{proof}
\bigskip

\begin{remark}\label{rmkP}
By \eqref{KvsH}, we have
$$P_{w, v}(q)= q^{\frac{l(w)}{2}}H_{w, v}(\sqrt q)=\tilde H_{w, v}(\sqrt q),  $$
hence previous corollary provides an iterative procedure that allows to compute  \textbf{both} the Kazhdan-Lusztig polynomials $P_{w, v}$ and the Laurent polynomials $D_{w,  v}$.
\end{remark}

\bigskip

\section{Bases for the Hecke algebra}

\medskip
As in the previous section we assume to have fixed an equivariant resolution
$\pi_{w}$ for any Schubert variety $X(w)$, $w\in W$ and we assume in addition
that the cohomology of the fibers of $\pi_w$ vanish in odd degrees (this property is satisfied by any reasonable resolution of Schubert varieties). As a consequence, we have that the coefficients of the polynomials
$\tilde F_{w, v}$ vanish in odd degree.
Furthermore, from the relations
\begin{equation*}
		\begin{cases}
			\tilde D_{w, u} = t^{l(w)-l(u)} \circ S \circ t^{l(u)-l(w)} \circ  U_{l(w)-l(u)}(\tilde R_{w, u})\\
			\tilde R_{w, u} := \tilde F_{w, u} - \sum_{u < v < w} \tilde D_{w, v} \cdot \tilde H_{v, u}
		\end{cases}
	\end{equation*}
one deduces immediately that the same holds true for the polynomials $\tilde D_{w, v}$. We define
\begin{equation}\label{Q}
Q_{w, v}(q)=\tilde F_{w, v}(\sqrt q)  \in \Z[q], 
\end{equation}
\begin{equation}\label{S}
S_{w, v}(q)=\tilde D_{w, v}(\sqrt q)  \in \Z[q], 
\end{equation}
for all $v, w \in W$ such that $v\leq w$.
Our aim in this section is to define a new basis for the Hecke algebra by means of the polynomials $Q_{w, v}(q)$. We start by recalling the definition of Hecke algebra.

Let $\hc$ be the \textit{Hecke algebra} of $W$ i.e. the algebra over $\Zq$ with basis elements 
\par\noindent $\{T_w\mid \,\, w\in W \}$ and relations \cite[Sec. 6.1]{BillLak}

\begin{equation*}
		\begin{cases}
	 		T_{s_i}T_w=T_{s_iw} \quad \text{if} \quad l(s_iw)>l(w),\\
			T_{s_i} T_{s_i}=(q-1) T_{s_i}+q T_{id}.
		\end{cases}
	\end{equation*}
The Hecke algebra is also equipped with the Kazhdan-Lusztig basis $\{C_w\mid \,\, w\in W \}$
where
\begin{equation}\label{KLbasis}
C_w= T_w + \sum_{v<w}P_{w, v}T_v
\end{equation}
where $P_{w, v}\in \Z[q]$ are the Kazhdan-Lusztig polynomials.

\begin{theorem}
 For any $w\in W$, let 
$$B_w= T_w + \sum_{v<w}Q_{w, v}T_v\in \hc.  $$ 
The polynomials $S_{w, v}(q)\in \Z[q]$ are the coefficients of $B_w$ with respect to the Kazhdan-Lusztig basis
$$B_w= C_w + \sum_{v<w}S_{w, v}C_v\in \hc.  $$

\end{theorem}
\begin{proof}
From \eqref{Eq:Eq1MainTh} and recalling $\tilde F_{w, u}=t^{l(w)}F_{w, u}$, we get
\begin{equation}\label{Eq:Eq1Th}
		\tilde F_{w, u}=t^{l(w)}F_{w, u}= \sum_{u\leq v \leq w} t^{l(w)-l(v)}D_{w, v} \cdot  t^{l(v)}H_{v, u}= \sum_{u\leq v \leq w} \tilde D_{w, v} \cdot  \tilde H_{v, u},
	\end{equation}
	where we have taken into account  \ref{cor:Utilde}.
Combining \ref{rmkP} with 	\eqref{Q} and \eqref{S} and evaluating in $t=\sqrt q$ the last equality, we get
\begin{equation}\label{Eq:Eq2Th}
		Q_{w, u}		= \sum_{u\leq v \leq w} S_{w, v}\cdot P_{v, u}\in \Z[q].
\end{equation}
Since the resolution $\pi: \tilde X \rightarrow X$ is equivariant, it must be an ismorphism over $\Omega(w)$, so we have $Q_{w, w}=1$ and 
$$B_w= T_w + \sum_{u<w}Q_{w, u}T_u=  \sum_{u\leq w}Q_{w, u}T_u=\sum_{u\leq w}\left(\sum_{u\leq v \leq w} S_{w, v}\cdot P_{v, u}\right)T_u=\sum_{v\leq w}S_{w, v}\left(\sum_{u\leq v }  P_{v, u}T_u \right).$$
Again, since $\pi_w$ an ismorphism over $\Omega(w)$ and we have $D_{w, w}= S_{w, w}=1$ (recall \eqref{decEx} and \eqref{defD}) and 
the last sum can be written as
$$\sum_{u\leq w }  P_{w, u}T_u+ \sum_{v< w}S_{w, v}\left(\sum_{u\leq v }  P_{v, u}T_u \right)$$
that coincides with
$$C_w + \sum_{v<w}S_{w, v}C_v$$
in view of \eqref{KLbasis}.

\end{proof}

\begin{remark}\label{last}
Theorem above shows that the transition matrix, from the   Kazhdan-Lusztig basis $\{C_w \mid  w\in W\}$ to the new one $\{B_w \mid  w\in W\}$, is triangular with coefficients $S_{w,  v}\in \Z[q]$. In this work we have set up an iterative procedure that allows
the computation of such a matrix and which \textit{does not require prior knowledge of the transition matrix from the Kazhdan-Lusztig basis to the standard one} $\{T_w \mid  w\in W\}$.
\end{remark}

\bigskip

\textbf{Conflicts of interest}

\medskip
The author has no conflict of interest to declare that are relevant to this article.


\begin{thebibliography}{20}



\bibitem{BBD} 
A.~A. Be\u{\i}linson, J.~Bernstein, and P.~Deligne, \emph{Faisceaux pervers},
  Analysis and topology on singular spaces, {I} ({L}uminy, 1981),
  Ast\'{e}risque, vol. 100, Soc. Math. France, Paris, 1982, pp.~5--171.

  \bibitem{BillLak} 
  Sara Billey and V.~Lakshmibai, \emph{Singular loci of {S}chubert varieties},
  Progress in Mathematics, vol. 182, Birkh\"{a}user Boston, Inc., Boston, MA,
  2000.

\bibitem{CiFrSe}
Francesca Cioffi, Davide Franco, and Carmine Sessa, \emph{An effective decomposition theorem for Schubert varieties}, J. Symbolic Comput. \textbf{121} (2024), 102238.
DOI: 10.1016/j.jsc.2023.102238

 \bibitem{DeC} Mark Andrea~A. de~Cataldo, \emph{Hodge-theoretic splitting mechanisms for projective maps}, J. Singul. {\bf 7} (2013), 134-156.
 
\bibitem{dCaMi2009} 
Mark Andrea~A. de~Cataldo and Luca Migliorini,  \emph{The decomposition theorem, perverse sheaves and the topology of
  algebraic maps}, Bull. Amer. Math. Soc. (N.S.) \textbf{46} (2009), no.~4,
  535--633.


\bibitem{Deo} Vinay V. Deodhar, \emph{A combinatorial setting for questions in Kazhdan-Lusztig theory}, Geom. Dedicata
\textbf{36} (1990),
  no.~1, 95--119.



\bibitem{dGeFr2017OnTheExistence}  
Vincenzo Di~Gennaro and Davide Franco, \emph{On the existence of a {G}ysin morphism for the blow-up of an
  ordinary singularity}, Ann. Univ. Ferrara Sez. VII Sci. Mat. \textbf{63}
  (2017), no.~1, 75--86.


  
 \bibitem{DGF2} Vincenzo Di~Gennaro and Davide Franco,  \emph{N\'eron-Severi group of a general hypersurface},
Commun. Contemp. Math.,  {\bf 19} (2017), no.~1,  1650004, 15 pp. 
DOI: 10.1142/S0219199716500048

\bibitem{dGeFr2019}
Vincenzo Di~Gennaro and Davide Franco, \emph{On a resolution of singularities with two strata}, Results Math.
  \textbf{74} (2019), no.~3, Paper No. 115, 22 pp. DOI: 10.1007/s00025-019-1040-9


\bibitem{dGeFr2020}
Vincenzo Di~Gennaro and Davide Franco, \emph{On the topology of a resolution of isolated singularities,
  {II}}, J. Singul. \textbf{20} (2020), 95--102.




\bibitem{Forum} Davide Franco, \emph{Explicit decomposition Theorem for special Schubert varieties}, Forum Math. \textbf{32} (2020), no.~2, 447-470.


\bibitem{GMP2}
Mark Goresky and Robert MacPherson,  \emph{Intersection homology. {II}}, Invent. Math. \textbf{72} (1983),
  no.~1, 77--129.



\bibitem{HTT}
Ryoshi Hotta, Kiyoshi Takeuchi, and Toshiyuki Tanisaki, \emph{{$D$}-modules,
  perverse sheaves, and representation theory}, Progress in Mathematics, vol.
  236, Birkh\"{a}user Boston, Inc., Boston, MA, 2008, Translated from the 1995
  Japanese edition by Takeuchi. \MR{2357361}
  
  
  \bibitem{Max2019}
Lauren\c{t}iu~G. Maxim, \emph{Intersection homology \& perverse sheaves---with
  applications to singularities}, Graduate Texts in Mathematics, vol. 281,
  Springer, Cham, 2019.
  
  \bibitem{Sai1986}
Morihiko Saito, \emph{Mixed {H}odge modules}, Proc. Japan Acad. Ser. A Math.
  Sci. \textbf{62} (1986), no.~9, 360--363.

  
  \bibitem{Wil2017}
Geordie Williamson, \emph{The {H}odge theory of the decomposition theorem}, no.
  390, 2017, S\'{e}minaire Bourbaki. Vol. 2015/2016. Expos\'{e}s 1104--1119,
  pp.~Exp. No. 1115, 335--367.

 












\end{thebibliography}
\end{document}